\documentclass[a4paper,leqno]{amsart}

\usepackage{amsmath}
\usepackage{amssymb}
\usepackage{amsfonts}
\usepackage{enumerate}
\usepackage{amsthm}
\usepackage{esint}
\usepackage{bbm}

\usepackage[dvips]{graphicx}

\newtheorem{teo}{Theorem}[section]
\newtheorem{lem}[teo]{Lemma}

\newtheorem{prop}[teo]{Proposition}

\theoremstyle{definition}

\newtheorem{oss}[teo]{Remark}

\newtheorem{assumption}{Assumption}
\newtheorem{claim}{Claim}

\renewcommand{\eqref}[1]{\textnormal{(\ref{#1})}}

\numberwithin{equation}{section}

\newcommand{\R}{\mathbb{R}}

\title[A counterexample for multiscale decompositions]{A counterexample to convergence for multiscale decompositions}

\author[S.~Rebegoldi]{Simone Rebegoldi}
\author[L.~Rondi]{Luca Rondi}

\address[S.~Rebegoldi]{Dipartimento di Scienze Fisiche, Informatiche e Matematiche, Universit\`a degli Studi di Modena e Reggio Emilia, Italy}
\email{simone.rebegoldi@unimore.it}

\address[L.~Rondi]{Dipartimento di Matematica, Universit\`a degli Studi di Pavia, Italy}
\email{luca.rondi@unipv.it}

\date{}

\begin{document}

\begin{abstract}
We discuss the convergence of the multiscale procedure by Modin, Nachman and Rondi, Adv. Math. (2019), which extended to inverse problems the 
multiscale decomposition of images by
Tadmor, Nezzar and Vese, Multiscale Model. Simul. (2004).
We show that, for the classical multiscale procedure, the multiscale decomposition might fail even for the linear case with a Banach norm as regularization.

\bigskip

\noindent\textbf{AMS 2020 MSC:} 
68U10 (primary);
65J22 (secondary)

\medskip

\end{abstract}

\maketitle

\setcounter{section}{0}
\setcounter{secnumdepth}{2}

\section{Introduction}

Let $H$ be a Hilbert space with scalar product $\langle\cdot,\cdot\rangle_H$ and norm $\|\cdot\|_H$. Let $X$ be a Banach space and let $\Lambda:X\to H$ be a bounded linear operator.
Let $F\subset X$ be a Banach space. We assume that for any $\lambda>0$ and any $\tilde{\Lambda}\in H$ the following minimization problem admits a solution
\begin{equation}\label{existenceminimizereq}
\min\left\{\lambda\|\tilde{\Lambda}-\Lambda(u)\|_H^{2}+\|u\|_F:\ u\in X\right\}.
\end{equation}
where $\|\cdot\|_F$ is referred to as the \emph{regularization}.

When $X=H=L^2(\R^2)$, $\Lambda$ is the identity and $F=BV(\R^2)$, in its homogeneous version, that is, $\|u\|_{BV(\R^2)}=TV(u)$, $TV$ denoting the Total Variation, \eqref{existenceminimizereq} becomes
\begin{equation}\label{existenceminimizereq1}
\min\left\{\lambda\|\tilde{\Lambda}-u\|_{L^2(\R^2)}^{2}+TV(u):\ u\in L^2(\R^2)\right\}.
\end{equation}
which is the classical Rudin-Osher-Fatem model for denoising.

If $\Lambda:L^2(\R^2)\to L^2(\R^2)$ is a blurring operator, for instance $\Lambda(u)=K\ast u$, where $K\in L^{\infty}(\R^2)$ with compact support represents the point-spread function, 
\eqref{existenceminimizereq} becomes
\begin{equation}\label{existenceminimizereq2}
\min\left\{\lambda\|\tilde{\Lambda}-\Lambda(u)\|_{L^2(\R^2)}^{2}+TV(u):\ u\in L^2(\R^2)\right\}.
\end{equation}
which is the corresponding problem for deblurring.

The efficacy of the denoising or deblurring method strongly depends on the choice of the parameter $\lambda$. Tadmor, Nezzar and Vese, in \cite{T-N-V} for denoising and in \cite{T-N-V-2}
for deblurring and other imaging applications, proposed to find the correct $\lambda$ by an iterative procedure which also induces a multiscale decomposition of the looked-for original image. Namely, the multiscale procedure is the following. For $\tilde{\Lambda}\in H$, fixed
positive parameters $\lambda_n$, $n\geq 0$, 
let $\sigma_0=u_0\in X$ solve
\begin{equation}\label{regularizedpbm}
\min\left\{\lambda_0\|\tilde{\Lambda}-\Lambda(u)\|_H^{2}+\|u\|_F:\ u\in X\right\}.
\end{equation}
Then by induction we define $\sigma_n$, $n\geq 1$, as follows. Let $u_n$ be a solution to
 \begin{equation}\label{inductiveconstr}\min\left\{\lambda_n\|\tilde{\Lambda}-\Lambda(\sigma_{n-1}+u)\|_H^{2}+\|u\|_F:\ u\in X\right\}\end{equation}
and $\sigma_n=\sigma_{n-1}+u_n$, that is,
 \begin{equation}\label{partialsum1}
 \sigma_n=\sum_{j=0}^nu_j\quad\text{for any }n\in\mathbb{N}.
 \end{equation}
 The sequence $\{\sigma_n\}_{n\geq 0}$ exists, but it might be not uniquely determined.

Since 
\begin{equation}
\|\tilde{\Lambda}-\Lambda(\sigma_n)\|_H\leq  \|\tilde{\Lambda}-\Lambda(\sigma_{n-1})\|_H
 \quad\text{for any }n\geq 1,
\end{equation}
there exists
$$\lim_n\|\tilde{\Lambda}-\Lambda(\sigma_n)\|_H\geq \inf\big\{\|\tilde{\Lambda}-\Lambda(\sigma)\|:\ \sigma\in X\big\}.$$

From a theoretical point of view, first we can ask if $\sigma_n$ is indeed a minimizing sequence, that is, if
\begin{equation}\label{minimizing}
\lim_n\|\tilde{\Lambda}-\Lambda(\sigma_n)\|_H= \inf\big\{\|\tilde{\Lambda}-\Lambda(\sigma)\|:\ \sigma\in X\big\}.
\end{equation}
If \eqref{minimizing} holds, we can ask if $\sigma_n$ is converging in $X$. In fact, if $\sigma_{n_k}\to \sigma_{\infty}$ in $X$ as $k\to +\infty$, then $\sigma_{\infty}$
solves the minimization problem
\begin{equation}\label{minexist}
\min\big\{\|\tilde{\Lambda}-\Lambda(\sigma)|_H:\ \sigma\in X\big\}.
\end{equation}
Hence, if \eqref{minimizing} holds, \eqref{minexist} admitting a solution is a necessary condition for the convergence, maybe up to subsequences,
of the sequence $\sigma_n$.

The first convergence result was obtained in \cite{T-N-V} for the denoising case. Namely, if the noisy image $\tilde{\Lambda}$ belongs to $BV(\R^2)$ (or to an intermediate space between $L^2(\R^2)$ and $BV(\R^2)$), then we have the following multiscale decomposition
\begin{equation}\label{TNVconv}
\tilde{\Lambda}=\lim_n\sigma_n=\sum_{j=0}^{+\infty}u_j\quad\text{in }L^2(\Omega).
\end{equation}
provided $\lambda_n$ goes to $+\infty$ fast enough, for instance when $\lambda_n=\lambda_02^n$ for any $n\geq 0$.
In \cite{M-N-R}, the convergence analysis for these multiscale procedures has been extended to linear or even nonlinear, when $\Lambda$ is not a linear operator, inverse problems and to other applications such as image registration. In \cite[Theorem~2.1]{M-N-R}
it was shown that
\eqref{minimizing} holds under extremely general assumptions, from which in particular it follows that \eqref{TNVconv} holds even for $\tilde{\Lambda}\in L^2(\R^2)$.
The convergence of the sequence $\sigma_n$ is much more difficult to establish. Let us consider the following stronger assumptions.

\begin{assumption}\label{strongerassum} We assume that $F$ is dense and compactly immersed in $X$. Moreover, we assume that 
\begin{equation}\label{lsccondition}
\|\cdot\|_F\text{ is lower semicontinuous on }X,\text{ with respect to the convergence in }X.
\end{equation}
\end{assumption}

\begin{assumption}\label{stronger}
There exists $\hat{\sigma}\in F$ such that
$$
\|\tilde{\Lambda}-\Lambda(\hat{\sigma})\|_H=
\min\big\{\|\tilde{\Lambda}-\Lambda(\sigma)\|_H:\ \sigma\in X\big\}.
$$
\end{assumption}

Under Assumption~\ref{strongerassum},  in \cite{M-N-R} a so-called \emph{tighter multiscale procedure} was developed, for which \eqref{minimizing} still holds, 
see \cite[Theorem~2.4]{M-N-R}. Moreover, this tighter version has the advantage that if we further have
Assumption~\ref{stronger}, then convergence, up to a subsequence, of $\sigma_n$ in $X$ holds true, see \cite[Theorem~2.5]{M-N-R}.
After \cite{M-N-R}, the multiscale procedure was further analyzed and extended, see 
\cite{L-R-V,K-R-W,M-N}.

The question whether, under Assumptions~\ref{strongerassum} and \ref{stronger}, the tighter multiscale procedure is really needed or the classical one is enough to guarantee convergence
of $\sigma_n$ remained open. Here we show through an example that for the classical multiscale procedure, even in the linear case and under Assumptions~\ref{strongerassum} and \ref{stronger}, convergence might fail, indeed we might have that $\lim_n\|\sigma_n\|_{X}=+\infty$. In this simple example, we have that $\lambda_n=\lambda_0M^n$, for any $n\geq 0$,
with $M$ large enough ($M\geq 6$ would suffice). A more involved example with the same underlying idea, extends the result to any $M\geq 2$, see \cite[Section~4]{R-R}. In \cite[Section~4]{R-R} another counterexample is presented for a linear, although quite strange, blurring operator on the line, that is, for signals.
\cite{R-R} contains also some positive results, namely that convergence holds when the regularization is given by the norm of a Hilbert operator.

\section{The counterexample}\label{examplesec}

We begin by describing some properties of the minimization problem \eqref{existenceminimizereq}.

\begin{oss}\label{uniquenessremark} Under Assumption~\ref{strongerassum}, \eqref{existenceminimizereq} admits a solution. In general, if $u_0$ solves \eqref{existenceminimizereq}, then
$u$ solves \eqref{existenceminimizereq} if and only if $\Lambda(u-u_0)=0$ and $\|u\|_F=\|u_0\|_F$.
Thus uniqueness holds either if $\Lambda$ is injective on $F$
or $F$ is a strictly convex Banach space. 
\end{oss}

Let $G=F^{\ast}$ and let $\|\cdot\|_{\ast}$ be its norm.
Let $\tilde{\Lambda}\in H$ and $\tilde{\Lambda}^{\ast}$ be the linear functional on $H$ associated to $\tilde{\Lambda}$.
We say that $\tilde{\Lambda}^{\ast}\circ \Lambda \in G$ if the functional $F \ni u\mapsto \langle \tilde{\Lambda},\Lambda(u)\rangle_{H}$ is bounded with respect to the $F$-norm on $u$, that is, \begin{equation}\label{starnorm}
\|\tilde{\Lambda}^{\ast}\circ \Lambda\|_{\ast}:=\sup\{\langle \tilde{\Lambda},\Lambda(u)\rangle_{H}:\ u\in F\text{ with }\|u\|_F=1\}<+\infty.
\end{equation}

Using Meyer's arguments in \cite{YMey}, it is well-known that the following holds.

\begin{prop}\label{charprop}
Let $\tilde{\Lambda}\in H$ and let $u_0$ be a solution to \eqref{existenceminimizereq}.
Then we have the following characterizations.
\begin{enumerate}[a\textnormal{)}]
\item $u_0$ is a minimizer if and only if 
\begin{equation}\label{char}
\|v_0^{\ast}\circ\Lambda\|_{\ast}\leq \dfrac1{2\lambda_0}\quad\text{and}\quad\langle v_0,\Lambda(u_0)\rangle_{H}=\dfrac1{2\lambda_0}\|u_0\|_F,\end{equation}
where $v_0=\tilde{\Lambda}-\Lambda(u_0)$ and $v_0^{\ast}$ is the linear functional on $H$ associated to $v_0$.
\item $u_0=0$ if and only if $\tilde{\Lambda}^{\ast}\circ \Lambda\in G$ and $\|\tilde{\Lambda}^{\ast}\circ \Lambda\|_{\ast}\leq \dfrac1{2\lambda_0}.$
\item If $\|\tilde{\Lambda}^{\ast}\circ \Lambda\|_{\ast}> \dfrac1{2\lambda_0}$, including when $\tilde{\Lambda}\circ\Lambda\notin G$, then \eqref{char} may be replaced by
\begin{equation}\label{char2}
\|v_0^{\ast}\circ\Lambda\|_{\ast}=\dfrac1{2\lambda_0}\quad\text{and}\quad\langle v_0,\Lambda(u_0)\rangle_{H}=\dfrac1{2\lambda_0}\|u_0\|_F>0.\end{equation}
\end{enumerate}
\end{prop} 

We now describe our counterexample.
Let $X=l_1$ and $H=l_2$. We consider
$$F=\left\{\gamma\in l_1:\ \|\gamma\|_F:=\sum_{n=1}^{+\infty}n|\gamma_n|<+\infty\right\}.$$
We have that Assumption~\ref{strongerassum} is satisfied. Note that
$$G=\left\{\kappa:\ \|\kappa\|_{\ast}=\sup_{n\in\mathbb{N}}\frac{|\kappa_n|}{n}<+\infty\right\}$$
where the duality is given by
$$\langle \kappa,\gamma\rangle=\sum_{n=1}^{+\infty}\kappa_n\gamma_n\quad\text{for any }\kappa\in G\text{ and }\gamma\in F.$$

We fix constants $M\geq 2$, $\alpha_0>0$, $c_0>0$, $b>0$, $0<\delta<1$.
For any $n\geq 0$ we let
$$\lambda_n=\alpha_0M^n.$$
For any $j\in \mathbb{N}$, let
$\displaystyle \eta_j=\sqrt{\frac{c_0}{M^j}}$ and $\mu_j=-\delta\eta_j.$
For any $j\geq 0$ and $n\geq 0$, let
\begin{equation}\label{ujdef}
u_j=\frac{b}{j+2}e_{j+2}\quad\text{and}\quad
\sigma_n=\sum_{j=0}^nu_j.
\end{equation}

We define the linear operator $\Lambda:l_1\to l_2$ as follows
\begin{equation}\label{Lambdadefinition}
\Lambda(\gamma)=\sum_{j=1}^{+\infty}\gamma_j\Lambda(e_j)\quad\text{for any }\gamma\in l_1,
\end{equation}
where
$$\Lambda(e_1)=\left(\sum_{m=2}^{+\infty}\eta_me_m\right)  +\mu_{2}e_2$$
and for any $j\geq 2$
$$\Lambda(e_j)=\frac{j}{b}\left(\eta_je_j+\mu_{j}e_{j}  - \mu_{j+1}e_{j+1}\right).$$

\begin{lem}\label{injlemma}
The operator $\Lambda$ defined in \eqref{Lambdadefinition} is a bounded linear operator from $X=l_1$ into $H=l_2$. Moreover, 
$\Lambda$ is injective on $l_1$, thus in particular on $F$.
\end{lem}

\begin{proof}
It is easy to show that $\Lambda$ is a bounded linear operator from $l_p$ to $l_2$ for any $1\leq p\leq +\infty$.
We call $b_j=b/j$, $j\geq 2$.
About injectivity, we claim that
$\tilde{\gamma}=\Lambda(\gamma)=0$, that is, $\tilde{\gamma}_j=0$ for any $j\in\mathbb{N}$, if and only if
$$\frac{\gamma_j}{b_j}=-\gamma_1\quad\text{for any }j\geq 2.$$
It follows that $\gamma_1=0$, thus $\gamma=0$, otherwise $\gamma=\gamma_1(1,-b/2,\ldots,-b/j,\ldots)\not\in l^1$.
We prove the claim by induction. We have $\tilde{\gamma}_1=0$ and
$$0=\tilde{\gamma}_2=(\eta_2+\mu_2)\gamma_1+(\eta_2+\mu_2)\eta_2\frac{\gamma_2}{b_2},$$
that is, $\gamma_2/b_2=-\gamma_1$. Let $j\geq 3$ and assume that $\gamma_{j-1}/b_{j-1}=-\gamma_1$. Then
$$0=\tilde{\gamma}_j=\eta_j\gamma_1+(\eta_j+\mu_j)\frac{\gamma_j}{b_j}-\mu_j\frac{\gamma_{j-1}}{b_{j-1}}$$
and the proof is concluded.\end{proof}

We fix $\tilde{\Lambda}=\Lambda(e_1)$. Hence Assumption~\ref{stronger} is satisfied. In fact, $e_1\in F$ and, by the injectivity of $\Lambda$, $e_1$ is the only solution to
$$0=\|\tilde{\Lambda}-\Lambda(e_1)\|_H=\min\{\||\tilde{\Lambda}-\Lambda(\gamma)\|_H:\ \gamma\in l_1\}.$$
Nevertheless, we have the following result.
\begin{teo}\label{examplethm}
Let $\displaystyle M\geq \frac{1}{3-2\sqrt{2}}$ and $\alpha_0>0$. Then
there exist suitable constants $c_0>0$, $b>0$, $0<\delta<1$ such that the sequence $\{\sigma_n\}_{n\geq 0}$ as in \eqref{ujdef} coincides with the
multiscale sequence $\{\sigma_n\}_{n\geq 0}$ given by \eqref{regularizedpbm}, \eqref{inductiveconstr} and \eqref{partialsum1}.
\end{teo}

\begin{oss}
By injectivity of $\Lambda$ and Remark~\ref{uniquenessremark}, the multiscale sequence is uniquely defined. Since $\lim_n\|\sigma_n\|_X= +\infty$,
 no subsequence of $\sigma_n$ converges, despite we are in a linear case and 
both Assumptions~\ref{strongerassum} and \ref{stronger} hold.
However, $\Lambda$ as a bounded linear operator from $l_2$ to $l_2$ is not injective anymore and actually
$\displaystyle \sigma_n\to \sigma_{\infty}=\sum_{j=2}\frac{b}{j}e_j$, where the series is in the $l_2$ sense, and, by the proof of Lemma~\ref{injlemma}, indeed we have $\Lambda(\sigma_{\infty})=\Lambda(e_1)$.
\end{oss}

\begin{proof} By
 Proposition~\ref{charprop}, it is enough to show that
for any $n\geq 0$ we have
\begin{multline}\label{tobeproved}
\|(\Lambda(e_1)-\Lambda(\sigma_n))\circ\Lambda\|_{\ast}\leq \dfrac1{2\lambda_n}\quad\text{and}\\\langle \Lambda(e_1)-\Lambda(\sigma_n),\Lambda(u_n)\rangle_{H}=\dfrac1{2\lambda_n}\|u_n\|_F,
\end{multline}
that is, that \eqref{char} holds for any $n\geq 0$.

For any $n\geq 0$, we call $n_1=n+2\geq 2$ and compute
$$\displaystyle \Lambda(\sigma_n)=\sum_{j=2}^{n_1}\frac{b}{j}\Lambda(e_j)=\left(\sum_{m=2}^{n_1}\eta_me_m\right)+\mu_2e_2-\mu_{n_1+1}e_{n_1+1},$$
therefore
$$
\Lambda(e_1)-\Lambda(\sigma_n)   =\left(\sum_{m=n_1+`}^{+\infty}\eta_me_m\right)+\mu_{n_1+1}e_{n_1+1}.
$$
For any $j\in \mathbb{N}$ and $n_1=n+2\geq 2$, we call
$$A_{j,n_1}=\frac1j\langle \Lambda(e_1)-\Lambda(\sigma_n),\Lambda(e_j)\rangle_{H}$$
so that
$$\|(\Lambda(e_1)-\Lambda(\sigma_n))\circ\Lambda\|_{\ast}=\sup_{j\in\mathbb{N}}|A_{j,n_1}|.$$
Condition \eqref{tobeproved}, hence Theorem~\ref{examplethm}, is an immediate consequence of the following.
\begin{claim}\label{crucialclaim} 
There exist suitable constants $c_0>0$, $b>0$, $0<\delta<1$ such that
for any $n\geq 0$, calling $n_1=n+2$, we have
\begin{equation}\label{crucialclaimprop}
|A_{j,n_1}|\leq A_{n_1,n_1}=\frac{1}{2\lambda_n}=\frac1{2\alpha_0M^{n}}=\frac{M^2}{2\alpha_0M^{n_1}}\quad\text{for any }j\in\mathbb{N}.
\end{equation}
\end{claim}

To prove Claim~\ref{crucialclaim}, we begin by computing $A_{j,n_1}$. First,
$$
0<A_{1,n_1}=
(\eta_{n_1+1}+\mu_{n_1+1})\eta_{n_1+1}
 +\left(\sum_{m=n_1+2}^{+\infty}\eta^2_m\right)\leq \frac{c_0}{M^{n_1+1}}\left(\frac{M}{M-1}\right).
$$
We have
$$A_{j,n_1}=0\quad\text{for any }2\leq j\leq n_1-1.$$
We have
$$
A_{n_1.n_1}=
-\mu_{n_1+1}(\eta_{n_1+1}+\mu_{n_1+1})=\frac{c_0}{bM^{n_1+1}}\delta(1-\delta).
$$
In order to have
$$0<A_{1,n_1}<A_{n_1,n_1}=\frac{M^2}{2\alpha_0M^{n_1}}=\frac{M^3}{2\alpha_0M^{n_1+1}}$$
it is enough to choose
$$b:=\delta(1-\delta)\frac{M-1}{2M}\quad\text{and}\quad c_0:=\frac{bM^3}{2\alpha_0\delta(1-\delta)}.$$
We note that $b$ depends on $M$ and $\delta$ only, whereas
$c_0$ depends on $M$, $\alpha_0$ and $\delta$ only.

We have that
$$A_{n_1+1.n_1}=\frac1b\left((\eta_{n_1+1}+\mu_{n_1+1})^2-\mu_{n_1+2}\eta_{n_1+2}\right)=\frac{c_0}{bM^{n_1+1}}\left((1-\delta)^2+\frac{\delta}M\right)
$$
and
\begin{multline*}
A_{n_1+s.n_1}=\frac1b\left(\eta_{n_1+s}(\eta_{n_1+s}+\mu_{n_1+s})-\mu_{n_1+s+1}\eta_{n_1+s+1}\right)\\
=\frac{c_0}{bM^{n_1+s}}  \left((1-\delta)+\frac{\delta}{M}\right)\leq \frac{c_0}{bM^{n_1+1}}\left(\frac{(1-\delta)}{M}+\frac{\delta}{M^2}\right)
 \quad\text{for any }s\geq 2.$$
\end{multline*}
We note that
$$\frac{(1-\delta)}{M}+\frac{\delta}{M^2}\leq (1-\delta)^2+\frac{\delta}M$$
if and only if
$$\delta^2-2\left(1-\frac1M+\frac1{2M^2}\right)\delta+1-\frac1M\geq 0 $$
which is true for any $\delta\in \R$ and any $M\geq 2$.
Hence, it is enough to show that for
some $\delta$, $0<\delta<1$, we have
\begin{equation}\label{tobeproved2}
(1-\delta)^2+\frac{\delta}M\leq \delta(1-\delta),\quad\text{that is,}\quad
2\delta^2+\left(\frac1M-3\right)\delta+1 \leq0.
\end{equation}
For
$$\delta:=\frac{3-\frac1M}{4},$$
we have that $0<\delta<1$ and the polynomial in \eqref{tobeproved2} is less than or equal to $0$ if and only if
$$\left(3-\frac1M\right)^2\geq 8,\quad\text{that is,}\quad M\geq \frac{1}{3-2\sqrt{2}}.$$
The proof is concluded.
\end{proof}

\noindent
\textbf{Acknowledgements.}
SR and LR acknowledge support by project
PRIN 2022 n.~2022B32J5C funded by MUR, Italy, and by the European Union – Next Generation EU.
SR is also supported by ...
LR is also supported by GNAMPA-INdAM through 2023 projects. The second author wishes to thank Cristiana Seibu for her kind invitation to give a talk at the Minisymposium 17 of IPMS2024, where this work was presented.


\begin{thebibliography}{99}

\bibitem{K-R-W}
S.~Kindermann, E.~Resmerita and T.~Wolf,
\emph{Multiscale hierarchical decomposition methods for ill-posed problems},
Inverse Problems \textbf{39} (2023) 125013 (36 pp).

\bibitem{L-R-V}
W.~Li, E.~Resmerita and L.~A.~Vese,
\emph{Multiscale Hierarchical Image Decomposition and Refinements}: \emph{Qualitative and
Quantitative Results},
SIAM J. Imaging Sci. \textbf{14} (2021) 844--877.

\bibitem{YMey}
Y.~Meyer,
\emph{Oscillating Patterns in Image Processing and Nonlinear Evolution Equations},
American Mathematical Society, Providence, 2001.

\bibitem{M-N}
T.~Milne and A.~Nachman,
\emph{An optimal transport analogue of the Rudin--Osher--Fatemi model and its corresponding multiscale theory},
 SIAM J. Math. Anal. \textbf{56} (2024) 1114--1148.

\bibitem{M-N-R}
K.~Modin, A.~Nachman and L.~Rondi,
\emph{A Multiscale Theory for Image Registration and Nonlinear Inverse Problems}, Adv. Math. \textbf{346} (2019) 1009--1066.

\bibitem{R-R}
S.~Rebegoldi and L.~Rondi,
\emph{On the optimality of convergence conditions for multiscale decompositions in imaging and inverse problems},
preprint (2024).

\bibitem{T-N-V}
E.~Tadmor, S.~Nezzar and L.~Vese,
\emph{A multiscale image representation using hierarchical} $(BV,L^2)$ \emph{decompositions},
Multiscale Model. Simul. \textbf{2} (2004) 554--579.

\bibitem{T-N-V-2}
E.~Tadmor, S.~Nezzar and L.~Vese,
\emph{Multiscale hierarchical decomposition of images with applications to deblurring, denoising and segmentation},
Commun. Math. Sci. \textbf{6} (2008) 281--307.

\end{thebibliography}
\end{document}